\UseRawInputEncoding
\documentclass[12pt]{article}

\usepackage{dsfont}
\usepackage{amsfonts,amsmath,amsthm}
\usepackage{algorithm}
\usepackage{algpseudocode}
\usepackage{xcolor}
\usepackage{graphicx}
\usepackage{stmaryrd}        % provides \llbracket and \rrbracket

\usepackage[paper=a4paper,dvips,top=2cm,left=2cm,right=2cm,
foot=1cm,bottom=4cm]{geometry}

\begin{document}
	\large
	
	\title{{Even Order Pascal Tensors are Positive Definite}}%\footnote{This work was partially supported by Hong Kong Innovation and Technology Commission (InnoHK Project CIMDA).}}
	\author{Chunfeng Cui\footnote{LMIB of the Ministry of Education, School of Mathematical Sciences, Beihang University, Beijing 100191 China.
			({\tt chunfengcui@buaa.edu.cn}).}
		\and { \
			Liqun Qi\footnote{Department of Applied Mathematics, The Hong Kong Polytechnic University, Hung Hom, Kowloon, Hong Kong.
				%Department of Mathematics, School of Science, Hangzhou Dianzi University, Hangzhou 310018 China
				({\tt maqilq@polyu.edu.hk}).}
		}
		\and {and \
			Yannan Chen\footnote{School of Mathematical Sciences, South China Normal University, Guangzhou 510631, China (\texttt{email: ynchen@scnu.edu.cn}). %This author was supported by the National Natural Science Foundation of China (12171168, 12071159).}
		}}
	}
	%This author's work was supported by Beijing Natural Science Foundation (Grant No. Z190002).}
	\date{\today}
	\maketitle
	
	\begin{abstract}
		
		{In this paper, we show that even order Pascal tensors are positive definite, and odd order Pascal tensors are strongly completely positive.  The significance of these is that our induction proof method also holds for some other families of completely positives tensors, whose construction satisfies certain rules, such an inherence property holds.  We show that for all tensors in such a family, even order tensors would be positive definite, and odd order tensors would be strongly completely positive, as long as the matrices in this family are positive definite.  In particular, we show that even order generalized Pascal tensors would be positive definite, and odd order generalized Pascal tensors would be strongly completely positive, as long as generalized Pascal matrices are positive definite.  We also investigate even order positive definiteness and odd order strongly completely positivity for  fractional Hadamard power   tensors.}  {Furthermore, we study   determinants of Pascal tensors.} We    prove that the determinant of the $m$th order two dimensional symmetric Pascal tensor is equal to the $m$th power of the factorial of $m-1$.

		\medskip

		% \medskip

	\end{abstract}

	\renewcommand{\Re}{\mathds{R}}
	\newcommand{\rank}{\mathrm{rank}}
	\newcommand{\X}{\mathcal{X}}
	\newcommand{\A}{\mathcal{A}}
	\newcommand{\I}{\mathcal{I}}
	\newcommand{\B}{\mathcal{B}}
	\newcommand{\C}{\mathcal{C}}
	\newcommand{\OO}{\mathcal{O}}
	\newcommand{\e}{\mathbf{e}}
	\newcommand{\0}{\mathbf{0}}
	\newcommand{\dd}{\mathbf{d}}
	\newcommand{\ii}{\mathbf{i}}
	\newcommand{\jj}{\mathbf{j}}
	\newcommand{\kk}{\mathbf{k}}
	\newcommand{\va}{\mathbf{a}}
	\newcommand{\vb}{\mathbf{b}}
	\newcommand{\vc}{\mathbf{c}}
	\newcommand{\vq}{\mathbf{q}}
	\newcommand{\vg}{\mathbf{g}}
	\newcommand{\pr}{\vec{r}}
	\newcommand{\pc}{\vec{c}}
	\newcommand{\ps}{\vec{s}}
	\newcommand{\pt}{\vec{t}}
	\newcommand{\pu}{\vec{u}}
	\newcommand{\pv}{\vec{v}}
	\newcommand{\pn}{\vec{n}}
	\newcommand{\pp}{\vec{p}}
	\newcommand{\pq}{\vec{q}}
	\newcommand{\pl}{\vec{l}}
	\newcommand{\vt}{\rm{vec}}
	\newcommand{\vx}{\mathbf{x}}
	\newcommand{\vy}{\mathbf{y}}
	\newcommand{\vu}{\mathbf{u}}
	\newcommand{\vv}{\mathbf{v}}
	\newcommand{\y}{\mathbf{y}}
	\newcommand{\vz}{\mathbf{z}}
	\newcommand{\T}{\top}
	\newcommand{\R}{\mathcal{R}}

	\newtheorem{Thm}{Theorem}[section]
	\newtheorem{Def}[Thm]{Definition}
	\newtheorem{Ass}[Thm]{Assumption}
	\newtheorem{Lem}[Thm]{Lemma}
	\newtheorem{Prop}[Thm]{Proposition}
	\newtheorem{Cor}[Thm]{Corollary}
	\newtheorem{example}[Thm]{Example}
	\newtheorem{remark}[Thm]{Remark}
	
	\section{Introduction}
	
	The symmetric Pascal matrix has the form $P = \left({(i+j-2)! \over (i-1)!(j-1)!}\right)$.  Its entries are the entries of the Pascal triangle, i.e., the coefficients of binomial expansion.   It has many interesting properties and applications in numerical analysis, filter design, image and signal processing, probability, combinatorics, numerical analysis, {chaotic system}, and electrical engineering {\cite{BP92, ES04, Le10, ZHBHZ22}.}    In 2016, Luo and Qi extended symmetric Pascal matrices to symmetric Pascal tensors and proved that they are completely positive tensors.  Symmetric Pascal tensors are also studied in \cite{QL17}.   An $m$th order $n$-dimensional Pascal tensor has the form
	\begin{equation} \label{pascal}
		{\cal P} = \left({(i_1+\dots i_m - m)! \over (i_1-1)!\dots (i_m-1)!}\right),
	\end{equation}
	where $i_1,\dots,i_m=1,\dots,n, m \ge 3, n\ge 2$.   In this paper, {for brevity,} we simply call $P$ a Pascal matrix, and $\cal P$ a Pascal tensor.
	
	Besides being a completely positive tensor, does a Pascal tensor have any other interesting properties?
	{Structured tensors are very important in tensor analysis and applications \cite{QL17}.
		As the nice structure of Pascal tensors, the study of Pascal tensors may serve as a prototype for studying more complicated structured tensors.  This is the motivation of this paper.
		
		For even order tensors, positive semi-definiteness, positive definiteness and sum-of-squares (SOS) are three most important properties.   Pascal tensors are completely positive tensors.  Even order completely positive tensors are SOS and positive semi-definite  tensors.  Hence, even order Pascal tensors are SOS and positive semi-definite tensors.
		{It is well known that Pascal matrices are positive definite.}
		Thus, a core problem for Pascal tensors is: if even order Pascal tensors are positive definite  or not.   {To the best of our knowledge,} this is an unsolved problem in structured tensor study. In the next section, by induction, we show that even order Pascal tensors are positive definite.
		
		For odd order tensors, strongly completely positiveness is somewhat a property corresponding to positive definiteness in the even order case {\cite{FNZ19, LQ16, YXH22}.}  In Section 3, we show that odd order Pascal tensors are strongly completely positive.
		
		Actually, these two properties are valid for broader families of completely positive tensors.  We state such a theorem in Section 4, and show that even order generalized Pascal tensors are positive definite, and odd order generalized Pascal tensors are strongly completely positive, as long as generalized Pascal matrices are positive definite.  We also investigate these two properties for the   fractional Hadamard power   tensors.}

	{In 2004, Edelman and Strang showed that} the determinant of a Pascal matrix is always $1$ \cite{ES04}.  {This is an exceptionally elegant property. A nature question is:} what is the determinant of a Pascal tensor?  The determinant of a symmetric tensor was introduced in \cite{Qi05}.     Suppose that an $m$th order $n$-dimensional real tensor is denoted as ${\cal A} = \left(a_{i_1\dots i_m}\right)$, where $i_1, \dots, i_m = 1, \dots, n$.  Let ${\vx} \in \Re^n$.   Then ${\cal A}{\vx}^{m-1}$ is a vector in $\Re^n$, whose $i$th component is
	$$\sum_{i_2, \dots, i_m = 1, \dots, n} a_{ii_2\dots i_m}x_{i_2}\dots x_{i_m}.$$
	The determinant of $\cal A$, denoted as det$({\cal A})$, is the resultant of ${\cal A}{\vx}^{m-1} = \mathbf 0$.  For the determinant of such a tensor, also see \cite{HHLQ13, QL17}.
	
	{In Sections 5 and 6,} we explore this problem.  Denote the determinant of the $m$th order $n$-dimensional Pascal tensor by $f(m, n)$.
	{If we count matrices as second order tensors, we have $f(2, n) \equiv 1$ for $n \ge 2$.}
	In {Section 5}, we explore the value of $f(m, 2)$.  We multiply the Sylvester-Pascal matrix of the $m$th order two dimensional Pascal tensor by a certain pivoting matrix.  Then we take LU-decomposition of the reshaped Sylvester-Pascal matrix.  By identifying the formulas of the diagonal entries of the upper triangular matrix of the LU decomposition, and considering the sign changes, we prove that $f(m, 2) = \left[(m-1)!\right]^m$.
	In {Section 6}, we explore the value of $f(m, n)$ for $n \ge 3$.   We make a conjecture that $f(m, n)$ is the $m$th power of an integer function, and is an integer factor of $f(m, n+1)$.  In particular, we make a conjecture on the formula of $f(m, 3)$ for $m \ge 3$.  This forms a challenge to people.
	%If we count matrices as second order tensors, we have $f(2, n) \equiv 1$ for $n \ge 2$.
	
	%In Section 6, we show that even order strongly completely positive tensors are positive definite, and odd order strongly completely positive tensors are strongly positive definite.  Thus, another research question on Pascal tensors is to check and show if they are strongly completely positive.

	{Some final remarks are made in Section 7.}
	
	In this paper, $f, g, h, i, j, k, m, n, r$ and $s$ are all positive integer valued.

	%{We review some basic definitions and properties of dual quaternions and dual quaternion matrices in Appendix A.}

	{\section{Even Order Pascal Tensors Are Positive Definite}
		
		Let $\A = \left(a_{i_1\dots i_m}\right)$ be an $m$th order $n$-dimensional symmetric tensor, where $m \ge 2$ is an even number.  Then $F_{\A}(\vx) = \A\vx^m$ defines a function $F_\A : {\mathbb R}^n \to {\mathbb R}$.  If $F_\A(\vx) \ge 0$ for $\vx \in {\mathbb R}^n$, then $\A$ is called a positive semi-definite tensor.   If $F_\A(\vx)$ can be expressed as SOS polynomials of $x_1, \dots, x_n$, then $\A$ is called an SOS tensor.  If  $F_\A(\vx) > 0$ for $\vx \in {\mathbb R}^n$ and $\vx \not = \0$, then $\A$ is called a positive definite tensor.  An SOS tensor is a positive semi-definite tensor, but not vice versa.   Similarly, a positive definite tensor is a positive semi-definite tensor, but not vice versa.
		These three properties are the most important properties of even order symmetric tensors {\cite{QL17}}.
		
		Let $\A$ be a symmetric tensor of order $m$ and dimension $n$.  If there are nonnegative vectors $\vu^{(1)}, \dots, \vu^{(r)} \in {\mathbb R}^n$ such that
		\begin{equation} \label{cp}
			\A = \left(\vu^{(1)}\right)^m + \dots + \left(\vu^{(r)}\right)^m,
		\end{equation}
		then $\A$ is called a completely positive tensor.  If furthermore $\left\{ \vu^{(1)}, \dots, \vu^{(r)} \right\}$ spans ${\mathbb R}^n$, then $\A$ is called a strongly completely positive tensor \cite{LQ16, QL17}.
		
		A complex number $\lambda$ is called an eigenvalue of $\A$ with an eigenvector $\vx$ if $\vx \in {\mathbb C}^n$, $\vx \not = \0$ and
		$$\A \vx^{m-1} = \lambda \vx^{[m-1]},$$
		where $\vx^{[m-1]} \in {\mathbb C}^n$ and its $i$th component is $x_i^{m-1}$.  If $\vx \in {\mathbb R}^n$, then $\lambda$ is real and called an H-eigenvalue of $\A$.  An even order symmetric tensor always has H-eigenvalues.   An even order symmetric tensor is positive semi-definite if and only if it has no negative H-eigenvalues.   An even order symmetric tensor is positive definite if and only if it has no non-positive H-eigenvalues.  The determinant of a symmetric tensor is equal to the product of all of its eigenvalues.
		These are the main contents of the spectral theory of symmetric tensors \cite{Qi05, QL17}.

		\begin{Thm}\label{Thm:even}
			Let $\cal P$ be an even order Pascal tensor with order $m$ and dimension $n$.   Then it is a positive semi-define tensor and  an SOS tensor.  Furthermore, it is a positive definite tensor as long as its determinant $f(m, n) \not = 0$.
		\end{Thm}
		\begin{proof}
			By Proposition 6.19 of \cite{QL17}, the Pascal tensor $\cal P$ is a completely positive tensor.   Now, assume that $\cal P$ has the form of (\ref{cp}).  Then
			\begin{equation} \label{FA}
				F_{\cal P}(\vx) = \sum_{k=1}^r \left(\left(\vu^{(k)}\right)^\top \vx\right)^m.
			\end{equation}
			This shows an even order {Pascal tensor} is an SOS tensor, hence a positive semi-definite tensor.
			
			{If in addition, $f(m, n) \not = 0$, then $\cal P$ does not have zero $H$-eigenvalues.} The last conclusion of this theorem follows from {the fact that a positive semi-define tensor is positive  define if and only if it does not have zero H-eigenvalues. This completes the proof.}
			%the main content of spectral theory of symmetric tensors stated before this theorem.
		\end{proof}
		
		The last conclusion of this theorem shows the relation between the positive definiteness problem and the determinant problem of Pascal tensors.
		
		%By the discussion in the last two sections, we are almost sure that the determinant of a Pascal tensor is a positive integer.   However, a strict proof is still needed.
		
		The following is a sufficient and necessary condition for an even order completely positive tensor to be positive definite.
		
		\begin{Thm} \label{Thm:evenPD}
			Suppose that $\A$ is a completely positive tensor of order $m$ and dimension $n$, where $m$ is even.   Then $\A$ is positive definite if and only if $\A$ is strongly completely positive.
		\end{Thm}
		\begin{proof}
			
			Suppose that $\A$ is completely positive with the form (\ref{cp}).
			
			Suppose that $\A$ is strongly completely positive. Then for any $\vx \not = \0$, in (\ref{FA}), there is at least one $k$ such that
			$$\left(\left(\vu^{(k)}\right)^\top \vx\right)^m > 0,$$
			as $\left\{ \vu^{(1)}, \dots, \vu^{(r)} \right\}$ spans ${\mathbb R}^n$.   This shows that $\A$ is positive definite.
			
			On the other hand, suppose that $\A$ is not strongly completely positive, i.e.,  $\left\{ \vu^{(1)}, \dots, \vu^{(r)} \right\}$ does not span ${\mathbb R}^n$.  Then, we may find $\vx \in {\mathbb R}^n$, $\vx \not = \0$, such that $\left(\vu^{(k)}\right)^\top \vx = 0$ for $k = 1, \dots, r$.  This results $F_\A(\vx) = 0$, i.e., $\A$ is not positive definite.
		\end{proof}
		
		This theorem may also be resulted by Theorem 3.9 of \cite{LQ16}, and the main content of spectral theory of symmetric tensors stated before Theorem \ref{Thm:even}.
		
		We now present the main result of this section.
		
		\begin{Thm} \label{EvenPD}
			Even order Pascal tensors are positive definite.
		\end{Thm}
		\begin{proof}
			Fix $n \ge 2$.  We prove this by induction on $m$.
			
			For $m=2$, we know that Pascal matrices are positive definite.
			
			Suppose that the $m$th order $n$-dimensional Pascal tensor is positive definite, where $m$ is even.
			Consider the $(m+2)$th order $n$-dimensional Pascal tensor $\cal P$.  Since $\cal P$ is a completely positive tensor, we have nonnegative vectors $\vu^{(1)}, \dots, \vu^{(r)} \in {\mathbb R}^n$ such that
			\begin{equation} \label{Pcp}
				{\cal P} = \left(\vu^{(1)}\right)^{m+2} + \dots + \left(\vu^{(r)}\right)^{m+2}.
			\end{equation}
			Now, fix the first two indices of the entries of $\cal P$ as $i_1 = i_2 \equiv 1$.  Then we obtain an $m$th order $n$-dimensional sub-tensor of $\cal P$.  Denote it as ${\cal P}_{11}$.  Then, we have
			$${\cal P}_{11} = \left(u_1^{(1)}\right)^2\left(\vu^{(1)}\right)^m + \dots + \left(u_1^{(r)}\right)^2\left(\vu^{(r)}\right)^m,$$
			i.e.,
			\begin{equation} \label{P11cp}
				{\cal P}_{11} = \left(\vv^{(1)}\right)^m + \dots + \left(\vv^{(r)}\right)^m,
			\end{equation}
			where
			$$\vv^{(k)} = \left(u_1^{(k)}\right)^{2 \over m}\vu^{(k)},$$
			are nonnegative vectors, for $k = 1, \dots, r$.  By the structure of Pascal tensors, ${\cal P}_{11}$ is nothing else, but the $m$th order $n$-dimensional Pascal tensor.  Then (\ref{P11cp}) is its form as a completely positive tensor.   By our induction assumption, ${\cal P}_{11}$ is positive definite.  By Theorem \ref{Thm:evenPD}, ${\cal P}_{11}$ is strongly completely positive.  Thus, $\left\{\vv^{(1)}, \dots, \vv^{(r)} \right\}$ spans ${\mathbb R}^n$.  Then, $\left\{\vu^{(1)}, \dots, \vu^{(r)} \right\}$ also spans ${\mathbb R}^n$.   This implies that $\cal P$ is strongly completely positive.  By Theorem \ref{Thm:evenPD}, $\cal P$ is positive definite.  This completes our induction proof.
		\end{proof}
		
		%\begin{Cor}
		%Suppose that {$\mathcal P$} is an even order Pascal tensor of order $m$ and dimension $n$, with $m, n \ge 2$.   Then the following three statements are equivalent to each other.
		
		%1. {$\mathcal P$}  is positive definite;
		
		%2. $f(m, n) \not = 0$;
		
		%3. {$\mathcal P$} is strongly completely positive.
		
		%\end{Cor}

		\section{Pascal Tensors Are Strongly Completely Positive}

		We have the following theorem.
		
		\begin{Thm} \label{scp}
			Pascal tensors are strongly completely positive.
		\end{Thm}
		\begin{proof}
			By Theorems  \ref{Thm:evenPD} and \ref{EvenPD}, even order Pascal tensors are completely positive.   We now consider odd order Pascal tensors.
			
			Suppose that $\cal P$ is the $(m+1)$th order $n$-dimensional Pascal tensor, where $m, n \ge 2$ and $m$ is even.  Assume that it is not strongly completely positive.  By Theorem 3.9 of \cite{LQ16}, $\cal P$ must have a zero H-eigenvalue.   Thus, we have $\vx \in {\mathbb R}^n$, $\vx \not = \0$, such that ${\cal P}\vx^m = \0$.  This implies that ${\cal P}_1 \vx^m = 0$, where ${\cal P}_1$ is the $m$th order $n$-dimensional tensor obtained by fixing the first index of the entries of $\cal P$ as $i_1 \equiv 1$.  By the structure of Pascal tensors, ${\cal P}_1$ is exactly the $m$th order $n$-dimensional Pascal tensor.   Then, by the definition of even order positive definite tensors, ${\cal P}_1 \vx^m = 0$ and $\vx \not = \0$ implies that ${\cal P}_1$ is not positive definite.   This conflicts with Theorem \ref{EvenPD}.  Hence, $\cal P$ must be strongly completely positive.
			%Since the symmetric Pascal matrix is positive definite, we have the second conclusion of this theorem.
		\end{proof}
		
		%\section{Strongly Completely Positive Tensors}
		
		%We have the following theorem.
		
	}

	{\section{An Inherence Property for Some  Completely Positive Tensor Families}	
		
		Actually, Pascal tensors are prototypes of some families of completely positive tensors.   The study in Sections 2 and 3 shows that since Pascal matrices are positive definite and the construction of Pascal tensors satisfies certain rules, even order Pascal tensors are positive definite, and odd order Pascal tensors are strongly completely positive.   The significance is that this inherence property  may hold for some other completely positive tensor families.   In this section, we show that for a certain completely positive tensor family, if its construction satisfies certain rules, then even order tensors in this family would be positive definite, and odd order tensors are strongly completely positive, as long as matrices in this family are positive definite.
		
		Actually, such an inherence property reminds us strong Hankel tensors \cite{QL17}.  In general, even order Hankel tensors may not be positive definite.   However, if the Hankel matrix, generated by the same vector which generates a Hankel tensor, is positive semi-definite, then that Hankel tensor is called a strong Hankel tensor.  We know that even order strong Hankel tensors are positive semi-definite and SOS tensors \cite{QL17}.
		
		We have the following inherence theorem for completely positive tensors.
		
		\begin{Thm} {\bf (Inherence Theorem)} \label{Thm:class_SCP}
			Suppose that we have a family $C$ of completely positive tensors.   Assume that $C$ satisfies the following three properties.
			\begin{itemize}
				\item[1.]   For $m, n \ge 2$, there is at least one member in $C$.
				
				\item[2.]  For $m = 2$, all the member matrices in $C$ are positive definite.
				
				\item[3.]  If we fix the first index of the entries of an $m$th order $n$-dimensional tensor $\A$ in $C$ as $i_1 = 1$, then we have a sub-tensor $\A_1$ of $\A$.  Then $\A_1 \in C$.
			\end{itemize}

			Then all the tensors in $C$ are strongly completely positive, and all the even order tensors in $C$ are positive definite.
		\end{Thm}
		{
			\begin{proof}
				The results may be obtained   through the same induction process as in {Theorems \ref{EvenPD} and \ref{scp}}. Hence, we omit the details for brevity.
			\end{proof}	
		}
		
		The followings are some examples to apply this theorem.
		
		{1. Generalized Pascal {tensors}.  Let  $\mathbf c=(c_1,\dots,c_n)\in\mathbb R_{+}^n$ be a nonnegative vector  and let the  $m$th order $n$-dimensional tensor $\A = \left(a_{i_1\dots i_m}\right)$ be defined by
			\begin{equation}
				a_{i_1\dots i_m} = \frac{\Gamma(c_{i_1}+c_{i_2}+\dots+c_{i_m}+1)}{\Gamma(c_{i_1}+1)\Gamma(c_{i_2}+1)\cdots \Gamma(c_{i_m}+1)},
			\end{equation}
			then $\A$ is a generalized Pascal tensor and $\mathbf c$ is   called a
			generating vector of the generalized Pascal tensor $\A$.   If $c_i=i-1$ for any $i=1\dots,n$, then $\A$ is the Pascal tensor  \cite{QL17}.
			
			\begin{Thm}\label{Thm:Generalized_Pascal}
				Suppose that $\mathbf c$ satisfies $c_1=0$ and the  generalized Pascal matrix is positive definite. Then the generalized Pascal tensors are strongly completely pos. Even order generalized Pascal tensors are positive definite.
			\end{Thm}
			\begin{proof}
				
				By Proposition~6.19 in \cite{QL17}, generalized Pascal tensors are   completely positive.
				
				Because there are  no points $z$ at which $\Gamma(z)=0$,  the generalized Pascal tensor class is nonempty for any $m,n\ge2$.
				Furthermore, by fixing $i_1=1$, we have $\A_1$ is   an $(m-1)$th order $n$ dimensional generalized Pascal tensor with the same generating vector $\mathbf c$. Then the results of this theorem follows  directly from Theorem~\ref{Thm:class_SCP}.
				
				This completes the proof.
			\end{proof}

			2. The   fractional Hadamard power   tensors.

			For any given nonnegative tensor $\A = (a_{i_1\dots i_m})$,  its
			fractional Hadamard power is defined by the form $\A^{\circ \alpha} = (a_{i_1\dots i_m}^{\alpha})$. For factional
			Hadamard powers, the complete positivity may not be preserved, as shown by a counterexample in Qi and Luo \cite{QL17}.

			\begin{Thm}
				Suppose that  $\alpha>0$  and $\mathbf c$ satisfy $c_1=0$ and the  fractional Hadamard power   of the generalized Pascal matrix is positive definite. Then the  fractional Hadamard power   of the  generalized Pascal tensor $\A^{\circ \alpha}$ is strongly completely positive.  If in addition, $m$ is  even, then $\A^{\circ \alpha}$ is  positive definite.
			\end{Thm}
			\begin{proof}
				By Corollary~6.22 in \cite{QL17}, the  fractional Hadamard power   of the  generalized Pascal tensor is   completely positive.
				Then the conclusions of this theorem follow from Theorems~\ref{Thm:Generalized_Pascal} and \ref{Thm:class_SCP}.
			\end{proof}
			
			\begin{Thm}
				Let $\A=(a_{i_1\dots i_m})\in S_{m,n}$ be a generalized Hilbert tensor defined by
				\begin{equation*}
					a_{i_1\dots i_m}=\frac{1}{i_1+\dots +i_m - m + c},
				\end{equation*}
				where $c>0$ is a positive number.
				Suppose that $\alpha>0$ and $c$ satisfy $c_1=0$ and the  fractional Hadamard power   of the generalized Hilbert matrix  is positive definite. Then the fractional Hadamard power   of the generalized Hilbert  tensor  $\A^{\circ \alpha}$ is strongly completely positive. If in addition, $m$ is  even, then  $\A^{\circ \alpha}$ is positive definite.
			\end{Thm}
			\begin{proof}
				The generalized Hilbert tensor is a Cauchy tensor with $c_i = i-1+\frac{c}{m}$ for any $i=1,\dots,n$.
				By Theorem~6.21 in \cite{QL17}, the fractional Hadamard powers   of the generalized Hilbert tensors are   completely positive.
				
				It follows from the definition that the generalized Hilbert tensor is well defined  for any $m,n\ge2$.
				Furthermore, by fixing $i_1=1$, we have $\A_1$ is   an $(m-1)$th order $n$ dimensional generalized Hilbert tensor with the same  parameter  $c$. Then the results of this theorem follow  directly from  Theorem~\ref{Thm:class_SCP}.
				
				This completes the proof.
			\end{proof}
			
		}
		
		%1. Generalized Hilbert tensors.   Let an $m$th order $n$-dimensional tensor $\A = \left(a_{i_1\dots i_m}\right)$ is defined by
		%$$a_{i_1\dots i_m} = {1 \over i_1+ \dots + i_m + c},$$
		%where $c$ is a real number, $c \ge -1$.   If $c = -1$, then $\A$ is the Hilbert tensor \cite{QL17}.
		%
		%\begin{Thm}
		%Generalized Hilbert tensors are strongly completely positive tensors.  Even order generalized Hilbert tensors are positive definite.
		%\end{Thm}
		%\begin{proof}
		%\end{proof}

	}
	
	{\section{Determinants: The Case That $n=2$}}

	By the Sylvester Formula \cite{GKZ94, Qi05} and (\ref{pascal}), for $m \ge 3$,  we have	$f(m, 2) =$ det$(S_m)$, {where
		\begin{eqnarray*}
			S_m	&= &\left(\begin{array} {cccccccccc}
				1 & a_1 & a_2& \cdots &a_{m-2} & a_{m-1} & 0 & 0  & \cdots & 0 \\
				0 & 1   &a_1& \cdots &a_{m-3}& a_{m-2} & a_{m-1} & 0 & \cdots & 0 \\
				\cdot & \cdot &\cdot &  \cdots & \cdot & \cdot & \cdot & \cdot & \cdots & \cdot \\
				0 & 0 &0&\cdots  &  1&  a_1 &a_2& a_3 & \cdots   & a_{m-1} \\
				1 & b_1 & b_2& \cdots &b_{m-2} & b_{m-1} & 0 & 0  & \cdots & 0 \\
				0 & 1   &b_1& \cdots &b_{m-3}& b_{m-2} & b_{m-1} & 0 & \cdots & 0 \\
				\cdot & \cdot &\cdot &  \cdots & \cdot & \cdot & \cdot & \cdot & \cdots & \cdot \\
				0 & 0 &0&\cdots  &  1&  b_1 &b_2& b_3 & \cdots   & b_{m-1} \\
			\end{array}\right),
		\end{eqnarray*}
		$a_i=\left({m-1 \atop i}\right)i!$ and $b_i=\left({m-1 \atop i}\right)(i+1)!$ for $i=1,\dots,m-1$.}
	%\begin{eqnarray*}
	%& & S_m \cr
	%&= &\left(\begin{array} {cccccccccc}
	%	1 & \left({m-1 \atop 1}\right) & & \dots & \left({m-1 \atop m-2}\right)(m-2)! & (m-1)! & 0 & 0  & \dots & 0 \\
	%	0 & 1 &  & \dots & \left({m-1 \atop m-3}\right)(m-3)! & \left({m-1 \atop m-2}\right)(m-2)! & (m-1)! & 0 & \dots & 0 \\
	%	\cdot & \cdot & \cdot & \dots & \cdot & \cdot & \cdot & \cdot & \dots & \cdot \\
	%	0 & 0 &  & & 1 & \left({m-1 \atop 1}\right) & \left({m-1 \atop 2}\right)2! & \cdot & \cdot & (m-1)!\\
	%	1 & \left({m-1 \atop 1}\right)2! &  & \dots & \left({m-1 \atop m-2}\right)(m-1)! & m! & 0 & 0  & \dots & 0 \\
	%	0 & 1 & & \dots & \left({m-1 \atop m-3}\right)(m-2)! & \left({m-1 \atop m-2}\right)(m-1)! & m! & 0 & \dots & 0 \\
	%	\cdot & \cdot & & \dots & \cdot & \cdot & \cdot & \cdot & \dots & \cdot \\
	%	0 & 0 & & \dots & 1 & \left({m-1 \atop 1}\right)2! & \left({m-1 \atop 2}\right)3! & \cdot & \cdot & m!
	%\end{array}\right),
	%\end{eqnarray*}
	We call the $2(m-1) \times 2(m-1)$ matrix $S_m$ a Sylvester-Pascal matrix.
	
	By direct computation, we have
	$$f(3,2) = \left|\begin{array}{cccc}
		1 & 2 & 2 & 0  \\
		0 & 1 & 2 & 2\\
		1 & 4 & 6 & 0\\
		0 & 1 & 4 & 6
	\end{array}\right| = 8 = (2!)^3,$$
	$$f(4,2)= \left|\begin{array}{cccccc}
		1 & 3 & 6 & 6 & 0 & 0 \\
		0 & 1 & 3 & 6 & 6 & 0 \\
		0 & 0 & 1 & 3 & 6 & 6 \\
		1 & 6 & 18 & 24 & 0 & 0 \\
		0 & 1 & 6 & 18 & 24 & 0 \\
		0 & 0 & 1 & 6 & 18 & 24
	\end{array}\right| = 6^4 = (3!)^4$$
	and
	$$f(5,2)= \left|\begin{array}{cccccc}
		1 & 3 & 6 & 6 & 0 & 0 \\
		0 & 1 & 3 & 6 & 6 & 0 \\
		0 & 0 & 1 & 3 & 6 & 6 \\
		1 & 6 & 18 & 24 & 0 & 0 \\
		0 & 1 & 6 & 18 & 24 & 0 \\
		0 & 0 & 1 & 6 & 18 & 24
	\end{array}\right| = 24^5 = (4!)^5$$
	
	From the above fact, we may think that $f(m, 2) = \left[(m-1)!\right]^m$.
	
	Aiming at proving this, we consider the LU decomposition of the Sylvester-Pascal matrix.
	We first define the row permutation matrix $P_m$  to reshape $S_m$ as follows,
	{\begin{eqnarray*}
			T_m \equiv  P_mS_m =	\left(\begin{array} {cccccccccc}
				1 & b_1 & b_2& \cdots &b_{m-2} & b_{m-1} & 0 & 0  & \cdots & 0 \\
				1 & a_1 & a_2& \cdots &a_{m-2} & a_{m-1} & 0 & 0  & \cdots & 0 \\
				0 & 1   &b_1& \cdots &b_{m-3}& b_{m-2} & b_{m-1} & 0 & \cdots & 0 \\
				0 & 1   &a_1& \cdots &a_{m-3}& a_{m-2} & a_{m-1} & 0 & \cdots & 0 \\
				\cdot & \cdot &\cdot &  \cdots & \cdot & \cdot & \cdot & \cdot & \cdots & \cdot \\
				0 & 0 &0&\cdots  &   1&  b_1 &b_2& b_3 & \cdots   & b_{m-1} \\
				0 & 0 &0&\cdots  &   1&  a_1 &a_2& a_3 & \cdots   & a_{m-1} \\
			\end{array}\right),
		\end{eqnarray*}
		where $a_i=\left({m-1 \atop i}\right)i!$ and $b_i=\left({m-1 \atop i}\right)(i+1)!$ for $i=1,\dots,m-1$.}
	%	&=&\left(\begin{array} {cccccccccc}
	%		1 & \left({m-1 \atop 1}\right)2! &  & \dots & \left({m-1 \atop m-2}\right)(m-1)! & m! & 0 & 0  & \dots & 0 \\
	%		1 & \left({m-1 \atop 1}\right) & & \dots & \left({m-1 \atop m-2}\right)(m-2)! & (m-1)! & 0 & 0  & \dots & 0 \\
	%		0 & 1 & & \dots & \left({m-1 \atop m-3}\right)(m-2)! & \left({m-1 \atop m-2}\right)(m-1)! & m! & 0 & \dots & 0 \\
	%		0 & 1 &  & \dots & \left({m-1 \atop m-3}\right)(m-3)! & \left({m-1 \atop m-2}\right)(m-2)! & (m-1)! & 0 & \dots & 0 \\
	%		\cdot & \cdot & & \dots & \cdot & \cdot & \cdot & \cdot & \dots & \cdot \\
	%		\cdot & \cdot & \cdot & \dots & \cdot & \cdot & \cdot & \cdot & \dots & \cdot \\
	%		0 & 0 & & \dots & 1 & \left({m-1 \atop 1}\right)2! & \left({m-1 \atop 2}\right)3! & \cdot & \cdot & m!\\
	%		0 & 0 &  & & 1 & \left({m-1 \atop 1}\right) & \left({m-1 \atop 2}\right)2! & \cdot & \cdot & (m-1)!
	%	\end{array}\right).
	%\end{eqnarray*}
	Then {the  leading principal minors are not zeros and} the LU decomposition of $T_m$ always exists.
	To move the $m$-th row to the first row {while keeping the order of the first to $(m-1)$-th rows unchanged}, it requires $(m-1)$-th row switches, and to move the $(m+1)$-th row to the third row {while keeping the order of the third to $m$-th rows unchanged}, it requires $(m-2)$-th row switches. We repeat this process. It requires totally  $\frac{m(m-1)}{2}$ times of row permutations to formulate $PS_m$ stated above. Therefore, ${\rm det}(S_m) = (-1)^{\frac{m(m-1)}{2}}{\rm det}(T_m)$.

	Let  $$T_m = L_mU_m,$$ where $L_m$ is a lower triangular matrix with all $1$ diagonal entries.   Then $f(m, 2)$ is equal to the product of the diagonal entries of the upper triangular matrix $U_m$ {multiplied by $(-1)^{\frac{m(m-1)}{2}}$.}   Here are the forms of $U_m$ for $m = 3, 4, 5$:
	$$U_3 = \begin{bmatrix} 1 & 4 & 6 & 0 \\ 0 & -2 & -4 & 0 \\ 0 & 0 & 2 & 6 \\ 0 & 0 & 0 & 2 \end{bmatrix},$$
	$$U_4 = \begin{bmatrix} 1 & 6 & 18 & 24 & 0 & 0 \\ 0 & -3 & -12 & -18& 0 & 0 \\ 0 & 0 & 2 & 12 & 24 & 0  \\ 0 & 0 & 0 & 6 & 18 & 0 \\ 0 & 0 & 0 & 0 & 6 & 24 \\ 0 & 0 & 0 & 0 & 0 & -6  \end{bmatrix}$$
	and
	$$U_5 = \begin{bmatrix} 1 & 8 & 36 & 96 & 120 & 0 & 0 & 0 \\ 0 & -4 & -24 & -72& -96 & 0 & 0 & 0 \\ 0 & 0 & 2 & 8 & 72 & 120 & 0 & 0 \\ 0 & 0 & 0 & 12 & 72 & 144 & 0 & 0 \\ 0 & 0 & 0 & 0 & 6 & 48 & 120 & 0 \\ 0 & 0 & 0 & 0 & 0 & -24 & -96 & 0 \\ 0 & 0 & 0 & 0 & 0 & 0 & 24 & 120 \\ 0 & 0 & 0 & 0 & 0 & 0 & 0 & 24 \end{bmatrix}.$$
	
	%$$U_3 = \begin{bmatrix} 1 & 2 & 2 & 0 \\ 0 & 2 & 4 & 0 \\ 0 & 0 & 2 & 6 \\ 0 & 0 & 0 & 2 \end{bmatrix},$$
	%$$U_4 = \begin{bmatrix} 1 & 3 & 6 & 6 & 0 & 0 \\ 0 & 3 & 12 & 18& 0 & 0 \\ 0 & 0 & 2 & 12 & 24 & 0  \\ 0 & 0 & 0 & 6 & 18 & 0 \\ 0 & 0 & 0 & 0 & 6 & 24 \\ 0 & 0 & 0 & 0 & 0 & 6  \end{bmatrix}$$
	%and
	%$$U_5 = \begin{bmatrix} 1 & 4 & 12 & 24 & 24 & 0 & 0 & 0 \\ 0 & 4 & 24 & 72& 96 & 0 & 0 & 0 \\ 0 & 0 & 2 & 6 & 0 & -24 & 0 & 0 \\ 0 & 0 & 0 & 12 & 72 & 144 & 0 & 0 \\ 0 & 0 & 0 & 0 & 6 & 48 & 120 & 0 \\ 0 & 0 & 0 & 0 & 0 & 24 & 96 & 0 \\ 0 & 0 & 0 & 0 & 0 & 0 & 24 & 120 \\ 0 & 0 & 0 & 0 & 0 & 0 & 0 & 24 \end{bmatrix}.$$
	
	We see that det$(T_m)$ = det$(U_m)$ is equal to the product of the diagonal entries of $U_m$.  Then we have the following {claim}.

	{
		\textbf{Claim 5.1} Suppose  the LU decomposition of $T_m=L_mU_m$, where $L_m=(l_{ij})\in\mathbb R^{2(m-1)\times 2(m-1)}$ and $U_m=(u_{ij})\in\mathbb R^{2(m-1)\times 2(m-1)}$. Then,   we have $u_{2i-1, 2i-1} = i!$ and $u_{2i, 2i} = (-1)^i\left(m-1 \atop i\right)i!$ for $i = 1, \dots, m-1$.
	}

 By the above discussion, we see that this conclusion is true for $m = 2, 3, 4, 5$.   Assume that $m \ge 6$.
		
		Write $U_m = (u_{ij})$, $T_m = (t_{ij})$ and $L_m = (l_{ij})$ for $i, j = 1, \dots, 2m-2$.   Then $l_{ii} = 1$, $l_{ij} = 0$ for $j > i$ and $u_{ij} = 0$ for $j < i$.   By $T_m = L_mU_m$, for $i, j = 1, \dots, 2m-2$, we have
		\begin{equation} \label{LU}
			t_{ij} = \sum_{k=1}^{2m-2}l_{ik}u_{kj} = \sum_{k=1}^{\min \{i, j\}} l_{ik}u_{kj}
		\end{equation}
		Letting $i=1$ in (\ref{LU}),  we have $u_{1j} = t_{1j}$ for $j = 1, \dots, 2m-2$.  This implies that $u_{11} = t_{11} = 1$.
		
		Now, letting $i\ge 2$ and $j=1$ in (\ref{LU}),  we have $t_{i1} = l_{i1}u_{11} = l_{i1}$.  However, $t_{21} = 1$ and $t_{i1} = 0$ for $i \ge 3$.  Thus, $l_{21} = 1$ and $l_{i1} = 0$ for $i \ge 3$.
		
		Consider $i=2$ and $j \ge 2$ in (\ref{LU}).   We have $t_{2j} = l_{21}u_{1j} + l_{22}u_{2j} = t_{1j} + u_{2j}$.  This implies that
		$$u_{2j} = t_{2j}-t_{1j} = \left({m-1 \atop j-1}\right)(j-1)!(1-j),$$
		for $j = 2, \dots, m$, and $u_{2j} = t_{2j}-t_{1j} = 0$ for $j=m+1, \dots, 2m-2$.   Then $u_{22} = 1-m$.
		
		For {$i\ge3$} and $j=2$, by (\ref{LU}), we have $t_{32} = l_{31}u_{12} + l_{32}u_{22} = l_{32} (1-m)$.   Thus, $l_{32} = {t_{32} \over 1-m}$.  Since
		$t_{32} = 1$, $l_{32} = {-1 \over m-1}$. {Similarly, we have $l_{42} = \frac{-1}{m-1}$ and $l_{i2}=0$ for any $i\ge 5$.}
		
		For $i=3$ and $j\ge 3$, by (\ref{LU}), we have $t_{3j} = l_{31}u_{1j} + l_{32}u_{2j} + l_{33}u_{3j} = {u_{2j} \over 1-m} + u_{3j}$, i.e.,
		$$u_{3j} = t_{3j} - {u_{2j} \over 1-m},$$
		$$u_{33} = t_{33} - {u_{23} \over 1-m} = 2(m-1) + {\left({m-1 \atop 2}\right)2! \times 2 \over 1-m} = 2.$$
		{For  $i\ge4$ and $j=3$, we could derive $l_{43}=\frac{3-m}{2}$, $l_{53}=\frac12$, $l_{63}=\frac12$, and  $l_{i3}=0$ for any $i\ge 7$.}

		By repeating this process, we could derive that for $i=1,\dots,m-1$, $j=1,\dots,2m-2$,
		\[
		u_{2i-1, j} = \left\{
		\begin{array}{cl}
			i!,& \text{if } j = 2i-1; \\
			\left(m-i \atop j-2i+1\right)(j-i+1)!, & \text{if } 2i\le j\le m+i-1; \\
			0, & \text{otherwise},
		\end{array}
		\right.
		\]		
		\[
		u_{2i,j} = \left\{
		\begin{array}{cl}
			(-1)^i\left(m-1 \atop j-i\right)\left(j-i \atop i \right)(j-i)!, & \text{if } 2i\le j\le m+i-1; \\
			0, & \text{otherwise},
		\end{array}
		\right.
		\]
		and for $i = 1, \dots, 2m-2$,
		\begin{equation}\label{equ:L}
			l_{ij} =  \left\{
			\begin{array}{cl}
				1 & \text{if } j=i;\\
				\frac{|X_{ij}|}{u_{11}\dots u_{jj}}, & \text{if } j+1\le i\le 2j; \\
				0, & \text{otherwise},
			\end{array}
			\right.
		\end{equation}
		where $X_{ij}$ is a $j \times j$ matrix,
		\[
		X_{ij} = \begin{bmatrix}
			u_{11} & u_{12} & \dots & u_{1,j-1}& u_{1j} \\
			0 & u_{22} & \dots & u_{2,j-1} & u_{2j}  \\
			\vdots &\vdots &\ddots &\vdots &\vdots \\
			0 &0& \dots & u_{j-1,j-1}  & u_{j-1,j} \\
			t_{i1} & t_{i2} & \dots & t_{i,j-1}& t_{ij}
		\end{bmatrix}.
		\]
		By matrix multiplication, we have $L_mU_m = T_m$.
		Then, for $i = 1, \dots, m-1$, $u_{2i-1, 2i-1} = i!$ and $u_{2i, 2i} = (-1)^i\left(m-1 \atop i\right)i!$.

		{
			By \eqref{equ:L}, the expression of the lower-triangular matrix $L$ is more complicated than that of $U$. We list the  upper   left corner  of $L$ as follows,
			\begin{eqnarray*}
				L=	\begin{bmatrix}
					1&&& &&&\\
					1&1 &&&&&\\
					&\frac{1}{u_{22}} & 1 &&&&\\
					&\frac{1}{u_{22}} &\frac{3-m}{u_{33}} & 1 &&&\\
					&& \frac1{u_{33}} & \frac{4-m}{u_{44}} & 1 & &\\
					&&  \frac1{u_{33}} &\frac{5-2m}{u_{44}}  & \frac{m^2-6m+11}{u_{55}}& 1  &\\
					&&&  \frac1{u_{44}} &\frac{7-m}{u_{55}}  & \frac{m^2-7m+18}{u_{66}}& 1  &\\
					&&&  \frac1{u_{44}} &\frac{8-2m}{u_{55}}  & \frac{3m^2-17m+26}{u_{66}}&  \frac{(m-5)(-m^2+5m-10)}{u_{77}}&\cdots   \\
					&&&   &\frac{1}{u_{55}}  & \frac{10-2m}{u_{66}}&  \frac{m^2-11m+46}{u_{77}} &\cdots\\
					&&&   &\frac{1}{u_{55}}  & \frac{11-3m}{u_{66}}&  \frac{3m^2-25m+58}{u_{77}} &\cdots \\
					&&&   &&\frac{1}{u_{66}}  & \frac{14-2m}{u_{77}}&    \cdots  \\
					&&&   &&\frac{1}{u_{66}}  & \frac{15-3m}{u_{77}}&    \cdots  \\
					&&&&   &&\frac{1}{u_{77}}  &    \cdots  \\
					&&&&   &&\frac{1}{u_{77}}  &    \cdots  \\    &  &&&&   & & \cdots
				\end{bmatrix}.
			\end{eqnarray*}
		}
		
			\begin{Thm}
			For $m \ge 2$, we have $f(m, 2) = \left[(m-1)!\right]^m$.
		\end{Thm}
		\begin{proof}
			{Based on Claim 5.1, we have}
		\begin{eqnarray*}
			f(m, 2) & = & {\rm det}(S_m) = (-1)^{m(m-1) \over 2} \times {\rm det}(T_m) = (-1)^{m(m-1) \over 2} \times \prod_{r=1}^{2m-2} u_{r, r} \\
			& = & (-1)^{m(m-1) \over 2} \times \prod_{i=1}^{m-1} (-1)^i \times \prod_{r=1}^{2m-2} |u_{r, r}| = \prod_{r=1}^{2m-2} |u_{r, r}|.
		\end{eqnarray*}
		Therefore,
		\begin{eqnarray*}
			& & f(m, 2) \\
			& =  &  \prod_{r=1}^{2m-2} |u_{r, r}| \\
			& = &  \left[\prod_{k=1}^{m-2}u_{2k-1, 2k-1}\left|{u_{2m-2k-2,2m-2k-2}}\right|\right]\times u_{2m-3, 2m-3} \times \left|u_{2m-2, 2m-2}\right|\\
			& = &  \left[\prod_{k=1}^{m-2} k! \times \left(m-1 \atop m-k{-1}\right)\times (m-k{-1})! \right] \times \left[(m-1)!\right]^2 \\
			& = & \left[(m-1)!\right]^m.
		\end{eqnarray*}
		
		This completes the proof.
	\end{proof}

	{\section{Determinants: The General Case}}
	
	The entries of the $m$th order $n$-dimensional Pascal tensor are the coefficients of the $n$-th power of an $m$-term multinomial.
	{The  determinant of the $m$th order $n$-dimensional Pascal tensor $\cal P$, $f(m, n)$, is equal to the multivariate resultant of $\mathcal Px^{m-1}=0$.}  We may use the Macaulay formula \cite{CLO98, DD01} to  calculate $f(m, n)$ for $m \ge 3$ and $n \ge 3$.  The Macaulay formula converts the computation of a resultant to the computation of the quotient of two matrix determinants.
	%{Specifically, let
	%\[\mathcal M=\{x_1^{i_1}\cdots x_n^{i_n}| i_1+\cdots+i_n=n(m-1)\}\]
	%be the set of $n$-dimensional monimials of degree $n(m-1)$. There are $\binom{nm-1}{n-1}$ terms in $\mathcal M$ in total. Furthermore, let
	%\[\mathcal M_1=\left\{\frac{\mu}{x_1^{m-1}} \,\mid\, x_1^{m-1}|\mu, \mu \in \mathcal M\right\},\]
	%and for $i=2,\dots,n$,
	%\[\mathcal M_i=\left\{\frac{\mu}{x_i^{m-1}} \,\mid\, x_i^{m-1}|\mu, \mu \in \mathcal M\slash \{x_j^{m-1}\nu_j | \nu_j\in\mathcal M_j, 1\le j\le i-1\}\right\}.\]
	%Then the Macaulay matrix is a $\binom{nm-1}{n-1}\times \binom{nm-1}{n-1}$ real matrix. Each column of  the Macaulay matrix  is corresponding to one term in $\mathcal M$. The rows of the Macaulay matrix  is corresponding to  the terms in $\mathcal M_1,\dots,\mathcal M_n$, sequentially.
	%For the row corresponding to $\mu\in\mathcal M_i$ and the column corresponding to $\nu\in \mathcal M$, we fill the coefficient of $\nu$ in the polynomial $\mu P_i$, where $P_i$ denotes the $i$-th polynomial in $\mathcal Px^{m-1}$.
	%}
	Thus, theoretically, we even do not know if $f(m, n)$ is an integer for $m \ge 3$ and $n \ge 3$.
	
	By computation via the Macaulay formula and (\ref{pascal}), we have the following results
	$$f(3, 2) ={ \left|\begin{array}{cccc}
			1 & 2 & 2 & 0  \\
			0 & 1 & 2 & 2\\
			1 & 4 & 6 & 0\\
			0 & 1 & 4 & 6
		\end{array}\right| \slash 1} = 8 = 2^3,$$
	\begin{eqnarray*}
		&&	f(3, 3)\\
		&=&{
			\left| \begin{array}{ccccccccccccccc}
				1 & 2 & 2 & 2 & 6 & 6 & 0 & 0 & 0 & 0 & 0 & 0 & 0 & 0 & 0 \\
				0 & 1 & 0 & 2 & 2 & 0 & 2 & 6 & 6 & 0 & 0 & 0 & 0 & 0 & 0\\
				0 & 0 & 1 & 0 & 2 & 2 & 0 & 2 & 6 & 6 & 0 & 0 & 0 & 0 & 0\\
				0 & 0 & 0 & 1 & 0 & 0 & 2 & 2 & 0 & 0 & 2 & 6 & 6 & 0 & 0\\
				0 & 0 & 0 & 0 & 1 & 0 & 0 & 2 & 2 & 0 & 0 & 2 & 6 & 6 & 0\\
				0 & 0 & 0 & 0 & 0 & 1 & 0 & 0 & 2 & 2 & 0 & 0 & 2 & 6 & 6\\
				0 & 1 & 0 & 4 & 6 & 0 & 6 & 24 & 30 & 0 & 0 & 0 & 0 & 0 & 0\\
				0 & 0 & 1 & 0 & 4 & 6 & 0 & 6 & 24 & 30 & 0 & 0 & 0 & 0 & 0\\
				0 & 0 & 0 & 1 & 0 & 0 & 4 & 6 & 0 & 0 & 6 & 24 & 30 & 0 & 0\\
				0 & 0 & 0 & 0 & 1 & 0 & 0 & 4 & 6 & 0 & 0 & 6 & 24 & 30 & 0\\
				0 & 0 & 0 & 0 & 0 & 1 & 0 & 0 & 4 & 6 & 0 & 0 & 6 & 24 & 30\\
				0 & 1 & 0 & 6 & 12 & 0 & 12 & 60 & 90 & 0 & 0 & 0 & 0 & 0 & 0\\
				0 & 0 & 1 & 0 & 6 & 12 & 0 & 12 & 60 & 90 & 0 & 0 & 0 & 0 & 0\\
				0 & 0 & 0 & 0 & 1 & 0 & 0 & 6 & 12 & 0 & 0 & 12 & 60 & 90 & 0\\
				0 & 0 & 0 & 0 & 0 & 1 & 0 & 0 & 6 & 12 & 0 & 0 & 12 & 60 & 90
			\end{array} \right|\slash
			\left| \begin{array}{ccc} 1 & 0 & 6\\0 & 1 & 2\\ 0 & 1 & 6 	\end{array} \right|}\\
		&=& 72^3,
	\end{eqnarray*}
	and
	$$f(3, 4) = 72^3 \times 80^9 = \left(72 \times 512000\right)^3.$$

	Then we have a conjecture.
	
	{\bf Conjecture 1} For $n \ge 2$, $f(3, n)$ is a cubic number.
	
	Write $f(3, n) = \left[g(3, n)\right]^3$.  We have $g(3, 2) = 2$, $g(3, 3) = 72$ and $g(3, 4) = 72 \times 512000$. Then $g(3, 3) = g(3, 2) \times 6^2$ and $g(3, 4) = g(3, 3) \times 80^3$.   Thus, we have a further conjecture.
	
	{\bf Conjecture 2} For $n \ge 2$, $g(3, n+1)$ can be divided by $g(3, n)$, i.e., $g(3, n) = g(3, n-1)h(3, n)$ for some positive integer valued function $h$.
	
	According the last section, we have $f(m, 2) = \left[(m-1)!\right]^m$.   By computation, we know that
	$$f(3, 3) = \left[ 2^3 \times 3^2 \right]^3,$$
	$$f(4, 3) = \left[ 2^5 \times 3^9 \times 5^2 \right]^4,$$
	$$f(5, 3) = \left[ 2^{21} \times 3^{13} \times 5^4 \times 7^2 \right]^5,$$
	$$f(6, 3) = \left[ 2^{27} \times 3^{21} \times 5^{15} \times 7^4 \right]^6,$$
	$$f(7, 3) = \left[ 2^{44} \times 3^{40} \times 5^{19} \times 7^6 \times 11^2 \right]^7,$$
	$$f(8, 3) = \left[ 2^{52} \times 3^{50} \times 5^{23} \times 7^{21} \times 11^4 \times 13^2 \right]^8,$$
	$$f(9, 3) = \left[ 2^{105} \times 3^{62} \times 5^{29} \times 7^{25} \times 11^6 \times 13^4 \right]^9.$$
	
	Therefore, in general, we have the following conjecture.
	
	{\bf Conjecture 3} For $m \ge 3$ and $n \ge 2$, we have $f(m, n) = [g(m, n)]^m$, $g(m, 2) = (m-1)!$, $g(m, n+1) = g(m, n)h(m, n+1)$, where $g$ and $h$ are positive integer valued functions.
	
	Then we have $g(m, 2) = (m-1)!$,
	$$h(3, 3) = 6^2,$$
	$$h(4, 3) = \left(2^2 \times 3^4 \times 5\right)^2,$$
	$$h(5, 3) = \left(2^9 \times 3^6 \times 5^2 \times 7 \right)^2,$$
	$$h(6, 3) = \left(2^{12} \times 3^{10} \times 5^7 \times 7^2 \right)^2,$$
	$$h(7, 3) = \left(2^{20} \times 3^{19} \times 5^9 \times 7^3 \times 11 \right)^2,$$
	$$h(8, 3) = \left(2^{24} \times 3^{24} \times 5^{11} \times 7^{10} \times 11^2 \times 13 \right)^2,$$
	$$h(9, 3) = \left(2^{49} \times 3^{30} \times 5^{14} \times 7^{12} \times 11^3 \times 13^2 \right)^2.$$
	
	%Based upon these observations, we have the following conjecture.
	
	{\bf Conjecture 4} For $m \ge 3$, we have $f(m, 3) = \left[(m-1)! (s(m))^2\right]^m$,  where $s$ is a positive integer valued function.
	
	Here, we have
	$$s(3) = 2 \times 3,$$
	$$s(4) = 2^2 \times 3^4 \times 5,$$
	$$s(5) = 2^9 \times 3^6 \times 5^2 \times 7,$$
	$$s(6) = 2^{12} \times 3^{10} \times 5^7 \times 7^2,$$
	$$s(7) = 2^{20} \times 3^{19} \times 5^9 \times 7^3 \times 11,$$
	$$s(8) = 2^{24} \times 3^{24} \times 5^{11} \times 7^{10} \times 11^2 \times 13,$$
	$$s(9) = 2^{49} \times 3^{30} \times 5^{14} \times 7^{12} \times 11^3 \times 13^2.$$
	
	Based on these computational results, we have the following conjecture.
	
	{\bf Conjecture 5} For $m \ge 3$, we have $s(m) = \left[(m-1)!\right]^{m-1} \times 3!! \times \dots \times (2m-3)!!$, i.e.,
	\begin{equation} \label{fm3}
		f(m, 3) = \left\{\left[(m-1)!\right]^{2m-3} \times \left[3!! \times \dots \times (2m-3)!!\right]^2 \right\}^m.
	\end{equation}
	
	The above computational results for $m = 3, \dots, 9$, confirm this conjecture.   It would be a big challenge to prove or to disprove (\ref{fm3}).
	
	\bigskip
	
	{
		\section{Conclusions and Discussions}
		In this paper, we showed the even order Pascal tensors are    positive definite and strongly completely positive. We also showed that odd order Pascal tensors are   strongly completely positive.  We further presented the determinants of 2-dimensional Pascal tensors.   These research topics are highly related with each other. These attempts   not only increase our knowledge on Pascal tensors, but also enhance our understanding of general  structured  tensors.   This is the significance of our exploration.
		
		{In this paper, we concluded that an inherence property for Pascal tensors may work for some other completely positive tensor families, such as
			generalized Pascal tensors and  fractional Hadamard power tensors.   More study on such tensors and some other possible completely positive tensor classes would also be interesting.
			
			Very recently, the concepts of positive semi-definiteness, SOS and positive definiteness were extended to odd order symmetric tensors in \cite{QC25}.  Let $\A = \left(a_{i_1\dots i_m}\right)$ be an $m$th order $n$-dimensional symmetric tensor, where $m \ge 2$ is an odd number.  Then $G_{\A}(\vx) = \A\vx^{m-1}$ defines a function $G_\A : {\mathbb R}^n \to {\mathbb R}^n$.  If $G_\A(\vx) \ge \0$ for $\vx \in {\mathbb R}^n$, then $\A$ is called a strongly positive semi-definite tensor.   If each component of $G_\A(\vx)$ can be expressed as SOS polynomials of $x_1, \dots, x_n$, then $\A$ is called a strongly SOS tensor.  If If $G_\A(\vx) > \0$ for $\vx \in {\mathbb R}^n$ and $\vx \not = \0$, then $\A$ is called a strongly positive definite tensor.   Odd order completely positive tensors are strongly positive semi-definite and SOS tensors.  Thus, odd order Pascal tensors are strongly positive semi-definite and SOS tensors.  On the other hand, though several other odd order completely positive tensor classes, such as odd order Hilbert tensors and Lehmer tensors, have been proved to be strongly positive definite, we still do not know if odd order Pascal tensors are strongly positive definite or not.   Hence, this will be the next challenging question for Pascal tensors.}
		
		We explored the value of $f(m, n)$ for $n \ge 3$.   There are still some conjectures left unanswered.  We made a conjecture that $f(m, n)$ is the $m$th power of an integer function, and is an integer factor of $f(m, n+1)$.  In particular, we made a conjecture on the formula of $f(m, 3)$ for $m \ge 3$.
		It would be intriguing to provide responses to them.

	}
	%{In this paper, we raised three research topics on Pascal tensors.  These three research topics are highly related with each other.
	
	%1. The determinants of Pascal tensors.  Some conjectures remain there.  Perhaps, a moderate target is to show that the determinants of Pascal tensors are not zeros.  We strongly believe that this claim is true, yet a strict proof is still needed.  If this is proved, then we may conclude even order Pascal tensors are positive definite.
	
	%2. Whether Pascal tensors are strongly completely positive tensors or not.   If this is true, then we may also conclude that even order Pascal tensors are positive definite.
	%In fact, we show that an even order completely positive tensor is positive definite if and only if it is strongly completely positive.
	
	%3. One may also try other possible {ways} to show that even order Pascal tensors are positive definite, and odd order Pascal tensors are strongly positive definite.
	
	%Successful of these attempts would not only increase our knowledge on Pascal tensors, but also enhance our understanding of general {structured} tensors.   This is the significance of our exploration.}
	
	\bigskip	
	
	%\bigskip
	
	{{\bf Acknowledgment}}
	This work was partially supported by   the R\&D project of Pazhou Lab (Huangpu) (Grant no. 2023K0603),  the National Natural Science Foundation of China (Nos. 12471282, 12131004, 12171168 and 12071159), and the Fundamental Research Funds for the Central Universities (Grant No. YWF-22-T-204).
	
	%Lie algebra, {and Professor Yuanhua Ni and his students for discussion on kinematics control}.   In particular, we are grateful to Professor Chengming Bai and his Ph.D. student Yuanchang Lin, whose note enabled the Lie algebra identification in Section 3.
	
	%and Zhongming Chen for the discussion on standard dual quaternion optimization, to Wei Li for the discussion on hand-eye calibration, to Jiantong Cheng for the discussion on SLAM, to Guyan Ni for introducing Jiantong Cheng to me, and to Chen Ouyang and Jinjie Liu for Figures 1 and 2.   I would like to thank two anonymous referees who carefully read my manuscript and gave very helpful comments.

	{{\bf Data availability} Data will be made available on reasonable request.

		{\bf Conflict of interest} The authors declare no conflict of interest.}

	%\section*{Compliance with ethical standards}
	%\bigskip
	
	%{\bf Conflicts of Interest} The author declares no conflict of interest.

	% \vspace{100pt}

\end{document}